\documentclass[12pt]{amsart}

\usepackage[left=2.2cm,tmargin=2.5cm,bmargin=2.5cm,right=2.2cm]{geometry}

\usepackage{amsmath,amssymb,physics,cprotect,hyperref}

\newcommand{\Gal}{\operatorname{Gal}}

\newtheorem{theo}{Theorem}

\newtheorem{prop}[theo]{Proposition}
\newtheorem{lemma}[theo]{Lemma}

\theoremstyle{remark}
\newtheorem{remark}[theo]{Remark}

\theoremstyle{definition}

\begin{document}

\title{An explicit lower bound for the unit distance problem}

\author{Will Sawin}

\maketitle

\begin{abstract} We show that there are sets of $n$ points in the plane with $n$ arbitrarily large that contain more than $n^{1.014}$ pairs of points separated by a distance exactly $1$. This improves on very recent work of a team at OpenAI, who proved the same result with an inexplicit exponent greater than $1$, drastically improving on the best previous lower bound and disproving a conjecture of Erd\H{o}s. The method is number-theoretic, relying on constructing algebraic number fields of large degree and small discriminant with many primes of small norm via a Golod-Shafarevich criterion argument.  \end{abstract}

We prove the following:

\begin{theo}\label{main} For $n$ arbitrarily large, there exists a set of points $U \subset \mathbb R^2$ such that $\# U = n$ and $\# \{ (v_1,v_2) \in U \mid \abs{v_1-v_2}=1\} \geq n^{1.014114}/C$ for an absolute constant $C$. \end{theo}

The best known upper bound is $\# \{ (v_1,v_2) \in U \mid \abs{v_1-v_2}=1\}  = O(n^{4/3})$, due to Spencer, Szemerédi, and Trotter~\cite{Spencer1983}.

Until very recently, the best lower bound known in Theorem \ref{main} was of the form $\# \{ (v_1,v_2) \in U \mid \abs{v_1-v_2}=1\} \geq n^{1+ \frac{c}{\log \log n}}$ for some $c>0$,  by Erd\H{o}s~\cite[Theorem 2]{Erdos1946}, who also conjectured that this was optimal~\cite[p. 267]{Erdos1994} (sometimes making the weaker conjecture that $\# \{ (v_1,v_2) \in U \mid \abs{v_1-v_2}=1\} \leq n^{1+o(1)}$~\cite[p. 62]{Erds1982}). Very recently, both conjectures were disproved~\cite{machinewriteup} by a team at OpenAI, consisting of Lijie Chen using an internal OpenAI model and Mark Sellke and Mehtaab Sawhney verifying correctness.
They showed a lower bound of the form $n^{1+\delta}$ for some $\delta>0$ not made explicit. It would be possible to make this explicit, but the value of $\delta$ obtained would likely be very small. A team of mathematicians produced a simplified version of this argument~\cite{humanwriteup}, and obtained $\delta \approx 6 \times 10^{-38}$. We provide a bound with an exponent $\delta$ that is explicit, and not exorbitantly small (differing from the upper bound by a factor of less than $24$), though certainly not optimal.

To do this, we make explicit and sharpen every step of the argument. Interestingly, this rarely makes the arguments much more complex: The total length of this paper is essentially the same as the original OpenAI writeup~\cite{machinewriteup}, though longer than the simplified version prepared by human authors~\cite{humanwriteup}.

The basic strategy is the same as the OpenAI team's proof: To construct a set of points in $\mathbb R^2$ with many unit distances, it suffices to construct a lattice, and a projection of the lattice to $\mathbb R^2$, with many short vectors whose projections to the plane have length $1$, as then the projection to $\mathbb R^2$ of the intersection of the lattice with a suitable ball will have many unit distances. (Erd\H{o}s's original argument used a lattice in $\mathbb R^2$, rather than a higher-dimensional lattice with a projection to $\mathbb R^2$.)

A useful source of lattices is the rings of integers of number fields, and these carry natural projections to the plane coming from the embeddings of the number field into the complex numbers. When the number field is a CM field $K$, containing a totally real subfield $F$, the squared length of the projection of a lattice vector to the plane lies in $F$. Since $F$ has smaller degree than $K$ and thus intuitively has fewer elements, it is possible for the projections of many vectors in $K$ to have the same length-squared in $F$, and then we can divide by this length to obtain many projections of lattice vectors with length $1$. (For other embeddings, the lengths are given by conjugates of the same elements of $F$, and this lets us control the total length of the vector as well.) One can interpret Erd\H{o}s's original argument as showing that this happens when $K= \mathbb Q(i)$ and $F=\mathbb Q$. Using a similar strategy, combined with a pigeonhole argument to deal with the class group, one can show that this happens for a general $K$ if $F$  has many small primes that split in $K$ and the class number of $K$ is not too large. If $F$ has many more small primes that split in $K$ than $\mathbb Q$ has small primes that split in $\mathbb Q(i)$, there is the potential for a better bound than Erd\H{o}s's, but one has to control the class number.

One has bounds for the class number in terms of the discriminant, so it therefore suffices to find totally real fields $F$ with CM extensions $K$ such that $F$ has many small primes that split in $K$ and the discriminant of $K$ is small. To make the construction work, one needs an infinite sequence of fields, and since one wants to take their discriminant small, this requires a sequence of growing degree. The best known construction of a sequence of number fields with increasing degree and discriminant not too large is the Golod-Shafarevich construction of $2$-class field towers of quadratic fields~\cite{GolodShafarevich}, which produces fields of discriminant exponential in the degree.\footnote{The original OpenAI solution used $3$-class field towers of cyclic cubic towers, but the argument was simplified to use $2$-class field towers after a suggestion of Victor Wang.} This construction uses the Golod-Shafarevich criterion~\cite{GolodShafarevich} to show that a certain Galois group is infinite to deduce that there are infinitely many distinct fields in the tower. However, these fields need not have many small primes, since the primes of $\mathbb Q$ may have large inertia degree (i.e. residue field extension degree) in these fields. Fortunately, it is possible to modify the construction to bound the inertia degree by applying the Golod-Shafarevich criterion to the Galois group after quotienting by Frobenius elements, obtaining an infinite group in which these Frobenius elements are trivial and thus the inertia degree is $1$. This argument was made, in a different context, by Hajir, Maire, and Ramakrishna~\cite{hajir2021shafarevich}.

Elements of the strategy have appeared before in work on different problems: The pigeonhole argument is similar to one used by Ellenberg and Venkatesh~\cite{Ellenberg2007} after a suggestion of Michel and Soundararajan to bound torsion subgroups of class groups, and the idea of constructing high-dimensional lattices with good properties from Golod-Shafarevich towers has been used multiple times~\cite{Lenstra1986Codes,LitsynTsfasman1987Constructive}. \cite{humanwriteup} contains a more detailed discussion of the context for this argument.

We improve the above argument in multiple ways: Rather than using the ring of integers as a lattice, we can work with any ideal. This makes the pigeonhole argument easier. Rather than using the class number, we can consider the relative class number, which is smaller. Rather than quotienting by Frobenius elements, we can quotient by a fixed power of the Frobenius element, forcing the inertia degree to be small instead of at most $1$, but making it much easier to apply the Golod-Shafarevich criterion (this argument was also made by Hajir, Maire, and Ramakrishna~\cite{hajir2021shafarevich}). We discuss other improvements after giving the proof.

This paper itself is likely far from optimized. It could be possible to improve the lower bound using techniques from analytic number theory, proving a more precise version of the key class group bound \cite[Corollary 3]{Louboutin2000} for the relevant number fields, using techniques from algebraic number theory, showing that there exist infinite towers of number fields satisfying a different set of local conditions, or using computational techniques, searching for better values of the parameters $T, S_{\mathbb Q}, k, R$.

Finally, we discuss two further questions, both suggested by Thomas Bloom: what is the density of the set of $n$ for which a set of points in Theorem \ref{main} exists, and what are the limits to the exponents that can be achieved by this argument?

The author was supported by NSF grant DMS-2502029 and was a Sloan Research Fellow while working on this manuscript and would like to thank Noga Alon, Thomas Bloom, Sebastien Bubeck, Timothy Gowers, Daniel Litt, Mehtaab Sawhney, Mark Sellke, Peter Sarnak, Arul Shankar, Kannan Soundararajan, Jacob Tsimerman, Victor Wang, and Melanie Matchett Wood for helpful conversations.

\section{Proofs}

In Lemma \ref{set-from-lattice}, we describe, in quantitative form, how a lattice with many short vectors that have the same length on projection to $\mathbb R^2$ gives a set of points in $\mathbb R^2$ with many unit distances.

To construct such a lattice, it suffices to choose an ideal in the ring of integers of a CM number field $K$ over a totally real field $F$ which contains many elements that have the same norm from $K$ to $F$. This is done in Lemma \ref{lattice-from-ideal}, resulting in Lemma \ref{set-from-ideal} which gives a simple bound for unit distances in terms of how many elements of an ideal of $K$ have a given norm.

 We can always find a suitable ideal when $F$ has many small primes that split in $K$.  This is shown in Lemma \ref{ideal-from-primes}.  Lemma \ref{galois-case} specializes Lemma \ref{ideal-from-primes} to the case when $K$ is Galois.  Lemma \ref{class-group-bound} gives a bound for the relative class number. Combining all these, Proposition \ref{main-sequence-criterion} gives a version of Theorem \ref{main} with an explicit value of $\delta$ if there exist infinitely many number fields satisfying certain criteria depending on a set of primes $S_{\mathbb Q}$. Lemma \ref{field-existence} constructs infinitely many such fields given another set of primes $T$, and we finally prove Theorem \ref{main} by finding explicit suitable sets $S_{\mathbb Q}$ and $T$.

\begin{lemma}\label{set-from-lattice} Let $d$ be a positive integer and $\Lambda$ be a lattice in $\mathbb R^{2d}$. Let $\norm{\cdot}$ be a norm on $\mathbb R^{2d}$, not necessarily Euclidean.  Let $\pi \colon \Lambda \to \mathbb R^{2}$ be an injective group homomorphism and let $\abs{\cdot}$ denote the Euclidean norm on $\mathbb R^2$. Let $\rho = \min ( \norm{v} \mid v \in \Lambda \setminus\{0\})$. Let $M$ be the number of vectors $v$ in $\Lambda$ such that $\norm{v} \leq 1$ and $ \abs{ \pi(v)}=1$. 

Then for each $R>1$ there exists a set $U \subseteq \mathbb R^2$ such that \[\# U \leq  \left( \frac{ 2R}{\rho}+1 \right)^{2d}\] and \begin{equation}\label{short-vector-ratio} \frac{\# \{ (v_1,v_2) \in U \mid \abs{v_1-v_2}=1\} }{ \# U} \geq \left(1-\frac{1}{R}\right)^{2d} M .\end{equation}
\end{lemma}

\begin{proof} For a point $w \in \mathbb R^{2d}$ and $r>0$, we let $B(r,w)$ be the ball of radius $r$ around $w$ defined using the norm $\norm{\cdot}$. We will choose $U$ to be $\pi ( B(R,w) \cap \Lambda)$ for a suitable point $w$.  We have
\[ \#U = \# (B(R,w)\cap \Lambda) \leq \left( \frac{ 2R+\rho}{\rho} \right)^{2d}= \left( \frac{ 2R}{\rho}+1 \right)^{2d} \]
because the balls of radius $\frac{\rho}{2}$ around the points of $B(R,w)\cap \Lambda$ do not overlap each other and are all contained in $B(R+\frac{\rho}{2},w)$, hence their total volume is at most the volume of $B(R+\frac{\rho}{2},w)$, which is $ \left( \frac{ 2R+\rho}{\rho} \right)^{2d}$ times the volume of a ball of radius $\frac{\rho}{2}$.

For any $ v_1 \in B(R-1,w)\cap \Lambda$, the number of $v_2$ in $B(R,w)\cap \Lambda$ with $\abs{\pi(v_1-v_2)}=1$ is at least $M$, since $v_1+v$ has this property for any $v\in \Lambda $ such that $\norm{v} \leq 1$ and $ \abs{ \pi(v)}=1$. Since $\pi$ is injective, every pair $(v_1,v_2)$ of this form in $B(R,w) \cap \Lambda$ gives a distinct pair in $U$, so that
\begin{equation}\label{short-vector-lattice-point}\frac{ \# \{ (v_1,v_2) \in U \mid \abs{v_1-v_2}=1\} }{\#U} \geq \frac{ \# (B(R-1,w)\cap \Lambda)}{\#(B(R,w)\cap \Lambda)} M .\end{equation}

For $r>0$, if we choose the point $w$ uniformly at random from a fundamental domain for $\Lambda$, the expected value of $\#(B(r,w) \cap \Lambda)$ is  $r^{2d}$ times the volume of the unit ball for $\norm{\cdot}$ divided by the volume of the fundamental domain of $\Lambda$. Hence the expected value of 
\[ \#(B(R-1,w) \cap \Lambda) - \left( 1-\frac{1}{R} \right)^{2d} \#(B(R,w)\cap \Lambda) \] is zero. Thus we can choose $w$ such that \begin{equation}\label{lattice-point-bound} \#(B(R-1,w) \cap \Lambda) - \left( 1-\frac{1}{R} \right)^{2d} \#(B(R,w)\cap \Lambda) \geq 0.\end{equation} Combining \eqref{short-vector-lattice-point} and \eqref{lattice-point-bound}, we obtain \eqref{short-vector-ratio}.
\end{proof}

We fix some notation that we will use for the remainder of the paper: Let $F$ be a totally real number field and let $K$ be a totally imaginary quadratic extension of $F$. Let $d=[F:\mathbb Q]$. Let $c$ denote the generator of $\Gal(K/F)$ which acts on $K$ by complex conjugation. Let $N_{K/F}$ be the map from fractional ideals of $K$ to fractional ideals of $F$ that sends an ideal $I$ to the ideal generated by $\beta c(\beta)$ for all $\beta \in I$.  Let $\Sigma_{F,\infty}$ be the set of infinite places of $F$. Let $h^-(K)= h(K)/h(F)$ be the relative class number.  Let $\Delta_K$ be the discriminant of $K$, $\Delta_F$ the discriminant of $F$, and $\operatorname{rd}_{K/F} =( \Delta_K/\Delta_F)^{1/d}$ be the relative root discriminant.

\begin{lemma}\label{index-formula} Let $I$ be a fractional ideal of $K$, and choose $\beta \in I$ and $\alpha \in N_{K/F}(I)$. Then we have 
\[ \frac{ \# (I/(\beta))}{ \#(N_{K/F}(I)/(\alpha))} =   \prod_{v \in \Sigma_{F,\infty}} \frac{ \abs{\beta}_v^2 }{ \abs{\alpha}_v}.\]
\end{lemma}

\begin{proof} Let $V_K$ be the set of places of $K$. To each place is associated an absolute value. For the nonarchimedean absolute value we will use the product formula normalization, while for the archimedean absolute value we will continue to use the standard normalization. This requires us to square the absolute value at complex places when we state the product formula. For $I$ a fractional ideal of $K$ and $v$ a non-archimedean place of $K$, we define $\abs{I}_v$ to be the $v$-adic absolute value of a generator of $I \otimes_{\mathcal O_K} \mathcal O_{K_v}$, and the same for $F$. For $\beta \in I$ we have by the Chinese remainder theorem and the product formula
\[ \# (I/(\beta)) =\prod_{v \in V_K\setminus \Sigma_{K,\infty}} \# ((I \otimes_{\mathcal O_K} \mathcal O_{K_v})/ (\beta \mathcal O_{K_v})) =\prod_{v \in V_K\setminus \Sigma_{K,\infty}} \abs{I}_v/ \abs{\beta}_v \] \[= \prod_{v \in V_K\setminus \Sigma_{K,\infty}} \abs{I}_v \cdot \prod_{v \in V_K\setminus \Sigma_{K,\infty}}\abs{\beta}_v^{-1}=  \prod_{v \in V_K\setminus \Sigma_{K,\infty}} \abs{I}_v \cdot \prod_{v \in \Sigma_{K,\infty} }\abs{\beta}_v^2 =  \prod_{v \in V_K\setminus \Sigma_{K,\infty}} \abs{I}_v \cdot \prod_{v \in \Sigma_{F,\infty}} \abs{\beta}_v^2.\]
Similarly, we have
\[ \# (N_{K/F}(I)/(\alpha)) =  \prod_{v \in V_F\setminus \Sigma_{F,\infty}} \# (( N_{K/F}(I)\otimes_{\mathcal O_F} \mathcal O_{F_v})/(\alpha \mathcal O_{F_v}))=  \prod_{v \in V_F\setminus \Sigma_{F,\infty}} \abs{N_{K/F}(I)}_v/\abs{\alpha}_v  \] \[ =  \prod_{v \in V_F\setminus \Sigma_{F,\infty}} \abs{N_{K/F}(I)}_v \cdot  \prod_{v \in V_F\setminus \Sigma_{F,\infty}}  \abs{\alpha}_v^{-1}=  \prod_{v \in V_F\setminus \Sigma_{F,\infty}}\abs{N_{K/F}(I)}_v \cdot \prod_{v \in \Sigma_{F,\infty} }\abs{\alpha}_v \]
and we have
\[  \prod_{v \in V_F\setminus \Sigma_{F,\infty}} \abs{N_{K/F}(I)}_v  = \prod_{v \in V_K\setminus \Sigma_{K,\infty}} \abs{I}_v\] since the contribution of each place on the left hand side matches the contribution of the places lying over it on the right hand side.

Combining the above, we obtain the statement. \end{proof}

\begin{lemma}\label{lattice-from-ideal} Let $I$ be a fractional ideal of $K$ and let $\alpha \in N_{K/F} (I)$ be an element.

Let $\Lambda$ be the ideal $I$ viewed as a lattice in $\prod_{v\in \Sigma_{F,\infty}}  K_v \cong \prod_{v\in \Sigma_{F,\infty}} \mathbb C \cong \mathbb R^{2d}$. For $x = (x_v)_{v \in \Sigma_{F,\infty}}\in \prod_{v\in \Sigma_{F,\infty}}  K_v$, let $\norm{x} = \sup_{v \in \Sigma_{F,\infty}} \frac{ \abs{x_v}}{\sqrt{\abs{\alpha}_v}}$. Let $\pi \colon I \to \mathbb R^2$ be the composition $I \to K \to K_v \cong \mathbb C \cong \mathbb R^2$, divided by $\sqrt{ \abs{\alpha}_v}$, for an arbitrarily chosen infinite place $v$. 

Then the minimum norm of a nonzero vector in $\Lambda$ is at least  $ (\# ( N_{K/F} (I)/ (\alpha)))^{-\frac{1}{2d}}$ and every element $\beta \in I $ such that $\beta c(\beta)=\alpha$ satisfies $\norm{\beta} =1$ and $\abs{\pi(\beta)}=1$. \end{lemma}

\begin{proof} Since the absolute value of a complex number is obtained by multiplying by its complex conjugate and taking the square root, for every $\beta\in I$ such that $\beta c(\beta)=\alpha$ and every $v\in \Sigma_{F,\infty}$ we have \[ \abs{\beta}_v = \sqrt{ \abs{ \beta c(\beta)}_v } =\sqrt{\abs{\alpha}_v}.\]
Thus
\[ \norm{\beta} =\sup_{v \in \Sigma_{F,\infty}} \frac{ \abs{\beta}_v}{\sqrt{\abs{\alpha}_v}}= \sup_{v \in \Sigma_{F,\infty}} 1 =1\]
and
\[ \abs{\pi(\beta)}=\frac{ \abs{\beta}_v}{\sqrt{\abs{\alpha}_v}}=1.\]

Furthermore, for an arbitrary nonzero vector in $\Lambda$, that is, for an arbitrary nonzero $\beta$ in $I$, by Lemma \ref{index-formula} we have
\[ \prod_{v \in \Sigma_{F,\infty}} \frac{ \abs{\beta}_v }{ \sqrt{ \abs{\alpha}_v}} =  \frac{ \sqrt{\# (I/(\beta))}}{\sqrt{  \# (N_{K/F}(I) /(\alpha))}} \geq \frac{1}{\sqrt{  \# (N_{K/F}(I)/(\alpha))}}\]
and hence
\[ \sup_{v \in \Sigma_{F,\infty}}\frac{ \abs{\beta}_v }{ \sqrt{ \abs{\alpha}_v}}  \geq (\# (N_{K/F}(I)/(\alpha)))^{-\frac{1}{2d}}\]
since $\#\Sigma_{F,\infty}= d$. \end{proof}

\begin{lemma}\label{set-from-ideal} Let $F$ be a totally real number field and $K$ a totally imaginary quadratic extension of $F$. Let $I$ be a fractional ideal of $K$ and let $\alpha \in N_{K/F}(I)$ be an element.  Let $M$ be the number of elements $\beta \in I $ such that $\beta c(\beta)=\alpha$.

Then for each $R>1$  there exists a set $U \subseteq \mathbb R^2$ such that \[\# U \leq  \left( 2R (\# (N_{K/F}(I)/(\alpha)))^{\frac{1}{2d}} +1 \right)^{2d}\] and \begin{equation}\frac{\# \{ (v_1,v_2) \in U \mid \abs{v_1-v_2}=1\} }{ \# U} \geq \left(1-\frac{1}{R}\right)^{2d} M .\end{equation}
\end{lemma}

\begin{proof} This follows immediately from combining Lemma \ref{set-from-lattice} and Lemma \ref{lattice-from-ideal}.\end{proof}

The remainder of the argument will be devoted to finding $F,K, I , \alpha$ with the number of elements $\beta \in I $ such that $\beta c(\beta)=\alpha$ as large as possible and  $\# (N_{K/F}(I)/(\alpha))$ as small as possible.

\begin{lemma}\label{modified-class-group} Let $G_K$ be the set of pairs of a fractional ideal $J$ of $K$ and a generator $u$ of $N_{K/F}(J)$, up to the equivalence relation $(J, u ) \sim ( (\alpha)J, \alpha c(\alpha) u)$ for all $\alpha \in K^\times$. Then  \begin{equation}\label{eq-mcg} \#G_K \leq  2^{d+1} h^{-}(K).\end{equation} \end{lemma}

The group $G_K$ may be described as the class group of a rank one torus over $F$, namely the kernel of the norm map from the Weil restriction of $\mathbb G_m$ from $K$ to $F$ to $\mathbb G_m$, and thus admits adelic and cohomological descriptions, but the elementary description here is most convenient for us.

\begin{proof} $G_K$ is a group under the group operation $(J_1, u_1) (J_2, u_2)=(J_1J_2,u_1u_2)$. We have an exact sequence

\[ \mathcal O_K^\times \to  \mathcal O_F^\times \to G_K \to \operatorname{Cl}(K) \to \operatorname{Cl}(F)  \]
where the maps $\mathcal O_K^\times \to \mathcal O_F^\times$ and $\operatorname{Cl}(K) \to \operatorname{Cl}(F)  $ are both the norm map $x \mapsto x c(x)$, the map $G_K \to \operatorname{Cl}(K) $ sends $(J,u)$ to $J$, and the map $ \mathcal O_F^\times \to G_K$ sends $u$ to $(1,u)$. It is straightforward to check that this sequence is exact.

This gives \begin{equation*}\begin{aligned} \#G_K= \# \operatorname{coker}( \mathcal O_K^\times \to  \mathcal O_F^\times ) & \#{\operatorname{ker} ( \operatorname{Cl}(K) \to \operatorname{Cl}(F))} \\ =  \#{ \operatorname{coker}( \mathcal O_K^\times \to  \mathcal O_F^\times )}& \frac{h(K)}{h(F)} \#{\operatorname{coker} ( \operatorname{Cl}(K) \to \operatorname{Cl}(F))}. \end{aligned}\end{equation*}

Since $\mathcal O_F^\times \subseteq \mathcal O_K^\times$  and the norm map restricted to $\mathcal O_F^\times$ is the square map, we have
\[  \#{ \operatorname{coker}( \mathcal O_K^\times \to  \mathcal O_F^\times )}  \leq  \# (  \mathcal O_F^\times/ (  \mathcal O_F^\times)^2) \leq 2^d \]
since $\mathcal O_F^\times \cong \mathbb Z^{d-1} \times \mathbb Z/2\mathbb Z$ by Dirichlet's unit theorem.

By class field theory, $\#{\operatorname{coker} ( \operatorname{Cl}(K) \to \operatorname{Cl}(F))} $ is the cokernel of the natural map from the Galois group of the maximal abelian unramified extension of $K$ to the Galois group of the maximal abelian unramified extension of $F$, and thus has size $2$ if $K/F$ is unramified at all finite places and $1$ otherwise, and in particular is at most $2$. Combining these, we obtain \eqref{eq-mcg}. \end{proof}

\begin{lemma}\label{ideal-from-primes} Let $S_F$ be a set of prime ideals of $\mathcal O_F$ and let $k$ be a function from $S_F$ to the positive integers. Assume that each prime ideal in $S_F$ is split in $K$. Then there exists a fractional ideal $I$ of $K$ and an element $\alpha \in N_{K/F}(I)$ such that the number of $\beta \in I$ such that $\beta c(\beta)=\alpha$ is at least \[ \frac{ \prod_{\mathfrak p \in S_F} (k(\mathfrak p)+1  )}{ 2^d h^{-}(K)} \] and  
\[ \# (N_{K/F}(I)/(\alpha)) = \prod_{\mathfrak p\in S_F} \# (\mathcal O_F/\mathfrak p)^{k(\mathfrak p)}.\] \end{lemma}

\begin{proof} Consider the set $L$ of ideals $J$ of $\mathcal O_K$ such that $N_{K/F}(J)=\prod_{\mathfrak p \in S_F} \mathfrak p^{k(\mathfrak p)}$. Then \[\# L = \prod_{\mathfrak p \in S_F} (k(\mathfrak p)+1  )\] since each prime $\mathfrak p$ in $S_F$ by assumption splits into two primes $\mathfrak p_1,\mathfrak p_2$ of $\mathcal O_K$, and any ideal $J_\mathfrak p$ of the form $\mathfrak p_1^j \mathfrak p_2^{k(\mathfrak p)-j}$ for $j\in \{0,\dots,k(\mathfrak p)\}$ will have $N_{K/F}(J_{\mathfrak p})= \mathfrak p^{k(\mathfrak p)}$, and then we can choose one $J_\mathfrak p$ for each $\mathfrak p$ and take the product.

If we fix one element $J_0 \in L$, then for any $J \in L$ we have 
\[ N_{K/F} ( J J_0^{-1})  = N_{K/F}(J)  N_{K/F}(J_0)^{-1}= (\prod_{\mathfrak p \in S_F} \mathfrak p^{k(\mathfrak p)}) (\prod_{\mathfrak p \in S_F} \mathfrak p^{k(\mathfrak p)})^{-1}  = (1) \] so that $J \mapsto ( J J_0^{-1}, 1)$ defines a map from $L$ to the set $G_K$ of Lemma \ref{modified-class-group}. By the pigeonhole principle, one fiber of this map must have cardinality at least $\frac{\#L}{\#G_K}$ . This means that there exists a fractional ideal $J_{\mathrm m}$ and $u $ a generator for $ N_{K/F}( J_{\mathrm m}) $ such that for at least $\frac{\#L}{\#G_K}$ ideals $J \in L$ there exists $\beta \in K$ with \begin{equation}\label{fiber-equation-1} \beta J_{\mathrm m} =  J J_0^{-1} \end{equation} and \begin{equation}\label{fiber-equation-2} u \beta c(\beta)=1.\end{equation}

Then we can take $I = J_0^{-1} J_{\mathrm m}^{-1} $ and $\alpha = u^{-1}$.  Since $u$ is a generator of $N_{K/F}(J_{\mathrm m})$, the element $\alpha$ is a generator of $N_{K/F}(J_{\mathrm m})^{-1}$. Since $N_{K/F}(J_0)$ is an ideal of $\mathcal O_F$,  the element $\alpha$ is contained in 
 \[N_{K/F}(I) = N_{K/F}(J_0)^{-1} N_{K/F}(J_{\mathrm m})^{-1}\] and we have
 \[ \# (N_{K/F}(I)/(\alpha)) = \#( (N_{K/F}(J_0)^{-1} N_{K/F}(J_{\mathrm m})^{-1})/ N_{K/F}(J_{\mathrm m})^{-1} )\] \[= \# (\mathcal O_F/ N_{K/F}(J_0)) = \#(\mathcal O_F/ \prod_{\mathfrak p \in S_F} \mathfrak p^{k(\mathfrak p)})=  \prod_{\mathfrak p\in S_F} \# (\mathcal O_F/\mathfrak p)^{k(\mathfrak p)}.\] 

Since $J$ is an ideal of $\mathcal O_K$, \eqref{fiber-equation-1} implies that $\beta J_{\mathrm m} \subseteq J_0^{-1}$ as fractional ideals and therefore $\beta \in J_{\mathrm m}^{-1} J_0^{-1}=I$.  Equation \eqref{fiber-equation-2} implies that $\beta c(\beta)=\alpha$.  The same $\beta$ cannot arise from two different ideals $J$ as we can recover $J$ from $\beta$ using \eqref{fiber-equation-1}, and each $J$ has at least two $\beta$ since if $\beta$ works then $-\beta$ does as well, so the number of $\beta \in I$ such that $\beta c(\beta)=\alpha$ is at least
\[ 2  \frac{ \#L}{\#G_K} \geq  \frac{ \prod_{\mathfrak p \in S_F} (k(\mathfrak p)+1  )}{ 2^d h^{-}(K)}\] by Lemma \ref{modified-class-group}. \end{proof}

\begin{lemma}\label{galois-case} Assume $K$ is a Galois extension of $\mathbb Q$, so that $F$ is Galois over $\mathbb Q$ also. Let $S_\mathbb Q$ be a set of prime numbers and let $k$ be a function from $S_\mathbb Q$ to the positive integers. For each prime  $p\in S_{\mathbb Q}$, assume that each prime of $F$ lying over $p$ splits in the extension from $F$ to $K$.  For each prime $p\in S_{\mathbb Q}$, let $e_p$ be the ramification index of $p$ in $F$ and let $f_p$ be the inertia degree of $p$ in $F$.

Then there exists a fractional ideal $I$ of $K$ and an element $\alpha \in N_{K/F}(I)$ such that the number of $\beta \in I$ such that $\beta c(\beta)=\alpha$ is at least \[ \frac{ \prod_{ p \in S_\mathbb Q} (k(p)+1  )^{\frac{d}{ e_p f_p}}}{ 2^d h^{-}(K)} \] and  
\[ \# (N_{K/F}(I)/(\alpha)) = \prod_{p \in S_\mathbb Q} p^{\frac{k(p) d}{e_p}}.\] \end{lemma}
\begin{proof} This follows from Lemma \ref{ideal-from-primes} by taking $S_F$ to be the set of primes lying over primes in $S_{\mathbb Q}$ and taking $k(\mathfrak p)=k(p)$ for $\mathfrak p$ lying over $p$, and recalling that the number of primes in a Galois extension of degree $d$ lying over a prime $p$ is $\frac{d}{ e_p f_p}$ and each of these primes has residue field of size $p^{f_p}$. \end{proof}

\begin{lemma}\label{class-group-bound} We have
\[ h^{-}(K) \leq 8  \operatorname{rd}_{K/F}^2 \left(\sqrt{\operatorname{rd}_{K/F} } \log(\operatorname{rd}_{K/F}) \frac{e}{4\pi} \right)^d . \] 
\end{lemma}

\begin{proof} Louboutin \cite[Corollary 3]{Louboutin2000} proved that
\[ h^{-}(K) \leq  2 Q_K w_K \sqrt{\Delta_K/\Delta_F} \left( \frac{e}{4\pi d} \log ( \Delta_K/\Delta_F) \right)^d = 2 Q_K w_K \operatorname{rd}_{K/F}^{\frac{d}{2} } \left( \frac{e \log \operatorname{rd}_{K/F}}{4\pi  }\right)^d \]
where $Q_K$ is the Hasse unit index of $K/F$ which is either $1$ or $2$ and $w_K$ is the number of roots of unity of $K$. Since $Q_K$ is either $1$ or $2$, it is at most $2$.

Recall that the root discriminant $\operatorname{rd}_K$ of a number field $K$ is $\Delta_K^{\frac{1}{[K:\mathbb Q]}}$. The root discriminant can only increase in a field extension, which implies that $\operatorname{rd}_{K/F} \geq \operatorname{rd}_K $ and also that $\operatorname{rd}_K \geq \operatorname{rd}_{ \mathbb Q(\mu_{w_K})}$ since $K$ certainly contains the cyclotomic field $\mathbb Q(\mu_{w_K})$.  The root discriminant of the cyclotomic field $\mathbb Q( \mu_{w_K})$ is \cite[Proposition 2.7]{Washington1997} \[\frac{w_K}{ \prod_{p \mid w_K} p^{ \frac{1}{p-1}} }\geq \sqrt{\frac{w_K}{2}}\]  so that $w_K$ is at most $2 \operatorname{rd}_{K/F}^2$.

Combining these, we obtain the statement. \end{proof}

\begin{prop}\label{main-sequence-criterion} Let $S_{\mathbb Q}$ be a set of prime numbers. Let $k,e$ and $f$ be functions from $S_\mathbb Q$ to positive integers. Let $\lambda>1$ be a real number.

Assume that there exist Galois CM fields $K$, of arbitrarily large degree, with totally real subfields $F$, such that  $\operatorname{rd}_{K/F}=\lambda$ and for each prime $p$ in $S_\mathbb Q$, every prime of $F$ lying over $p$ splits in the extension from $F$ to $K$, and $e(p)$ is the ramification index of $p$ in $F$, and $f(p)$ is at least the inertia degree of $p$ in $F$.

Then for any $R>1$, for $\delta$ defined by
\begin{equation}\label{delta-value} \delta=  \frac{  \log \left(1-\frac{1}{R}\right) + \frac{\log \left(\frac{2\pi}{e} \right)}{2} +\sum_{p\in S_{\mathbb Q}} \frac{1}{2e(p)f(p) } \log (k(p)+1)  - \frac{\log \lambda}{4} - \frac{\log \log \lambda}{2}}{ \log \left( 2R \prod_{p \in S_\mathbb Q} p^{\frac{k(p) }{2e(p)}} +1 \right)}\end{equation}
there are sets $U$ in the plane with $\#U$ arbitrarily large and \[ \# \{ (v_1,v_2) \in U \mid \abs{v_1-v_2}=1\} \geq  \frac{(\#U)^{1+\delta}}{8\lambda^2 } .\]

\end{prop}

\begin{proof} If $\delta \leq 0$ then the statement is trivial, so we may assume $\delta>0$.

We apply Lemma \ref{galois-case} to produce a fractional ideal $I$ and element $\alpha$ and then apply Lemma \ref{set-from-ideal}. We conclude that there exists a set $U \subseteq \mathbb R^2$ such that \[\# U \leq  \left( 2R (\# (N_{K/F}(I)/(\alpha)))^{\frac{1}{2d}} +1 \right)^{2d} \leq \Bigl( 2R \prod_{p \in S_\mathbb Q} p^{\frac{k(p) }{2e_p}} +1 \Bigr)^{2d} = \Bigl( 2R \prod_{p \in S_\mathbb Q} p^{\frac{k(p) }{2e(p)}} +1 \Bigr)^{2d} \] and  \[  \frac{\# \{ (v_1,v_2) \in U \mid \abs{v_1-v_2}=1\} }{ \# U} \geq \left(1-\frac{1}{R}\right)^{2d} M  \geq \left(1-\frac{1}{R}\right)^{2d} \frac{ \prod_{ p \in S_\mathbb Q} (k(p)+1  )^{\frac{d}{ e_p f_p}}}{ 2^d h^{-}(K)} \] \[\geq \left(1-\frac{1}{R}\right)^{2d} \frac{ \prod_{ p \in S_\mathbb Q} (k(p)+1  )^{\frac{d}{ e(p) f(p)}}}{  2^d \cdot 8 \lambda^2 \left(\sqrt{\lambda  } \log(\lambda) \frac{e}{4\pi} \right)^d  }\] applying Lemma \ref{class-group-bound} and using $e_p=e(p)$ and $f_p \leq f(p)$. Thus
\[    \frac{  \log( 8\lambda^2  \# \{ (v_1,v_2) \in U \mid \abs{v_1-v_2}=1\})}{\log \#U} \geq 1 + \frac{\log\left(  \left(1-\frac{1}{R}\right)^{2d} \frac{ \prod_{ p \in S_\mathbb Q} (k(p)+1  )^{\frac{d}{ e(p) f(p)}}}{ 2^d \left(\sqrt{\lambda} \log(\lambda) \frac{e}{4\pi} \right)^d }  \right)}{\log \left( \left( 2R \prod_{p \in S_\mathbb Q} p^{\frac{k(p) }{2e(p)}} +1 \right)^{2d}\right)}\]
\[ = 1 + \frac{  \log \left(1-\frac{1}{R}\right) + \frac{\log \left(\frac{2\pi}{e} \right)}{2} +\sum_{p\in S_{\mathbb Q}} \frac{1}{2e(p)f(p) } \log (k(p)+1)  - \frac{\log \lambda}{4} - \frac{\log \log \lambda}{2} }{ \log \left( 2R \prod_{p \in S_\mathbb Q} p^{\frac{k(p) }{2e(p)}} +1 \right)} = 1+\delta.\]
Furthermore, since $\delta>0$, the lower bound for  $  \frac{\# \{ (v_1,v_2) \in U \mid \abs{v_1-v_2}=1\} }{ \# U} $ is growing exponentially in $d$ and thus we can take $  \frac{\# \{ (v_1,v_2) \in U \mid \abs{v_1-v_2}=1\} }{ \# U} $ arbitrarily large. Since this quantity is certainly at most $\#U$, we have that $\#U$ is arbitrarily large. \end{proof}

The next two lemmas will produce suitable fields $F$ and $K$ to apply Proposition \ref{main-sequence-criterion}. The strategy is to start with a quadratic field $Q$ and choose $K$ and $F$ to be extensions of $Q$ unramified at all finite places. We will choose $f(p)=2$ for all $p$ in $S_\mathbb Q$, so that to apply Proposition \ref{main-sequence-criterion}, we need to check that there are infinitely many totally real fields $F$ that are everywhere unramified extensions of $Q$ such that for each prime $p$ in $S_{\mathbb Q}$, the inertia degree of $p$ in $F$ is at most $2$.  This is done in Lemma \ref{field-existence} using Galois-group calculations in Lemma \ref{group-infinite}.

\begin{lemma}\label{group-infinite} Let $T$ be a finite set of odd primes and $S_{\mathbb Q}$ a finite set of prime numbers. Assume that the number of elements of $T$ which are congruent to $3$ modulo $4$ is odd. Let $Q = \mathbb Q( \sqrt{ \prod_{q\in T} q })$.  Let $G$ be the Galois group of the composition of all Galois extensions of $Q$, with degree a power of $2$, that are everywhere unramified, totally real, and such that for all primes $\mathfrak p$ of $Q$ lying over $p \in S_{\mathbb Q}$, the inertia degrees of $\mathfrak p$ in the extensions is at most $2$, and is at most $1$ if $p$ is inert in $Q$.

Let $d(G)$ be the minimal number of generators for $G$ as a pro-$2$ group and $r(G)$ the minimal number of relations of $G$ when expressed by the minimal number of generators. Then:

\begin{enumerate} 
\item The field  $\mathbb Q( \{\sqrt{q} \mid q \in T\})$ is a Galois extension of $Q$, with Galois group $(\mathbb Z/2\mathbb Z)^{ \#T-1}$, that is everywhere unramified and totally real. For all primes $\mathfrak p$ of $Q$ lying over $p \in S_{\mathbb Q}$, the inertia degree of $\mathfrak p$ in the extension $\mathbb Q( \{\sqrt{q} \mid q \in T\})$  is at most $2$, and is at most $1$ if $p$ is inert in $Q$.
\item We have \[ d(G) \geq \#T -1 .\]
\item We have \[ r(G) \leq d(G) +  \# S_\mathbb Q + \# \{p \in S_\mathbb Q \mid p \textrm{ split in }Q \}+2 .\]
\item If  \begin{equation}\label{infinitude-criterion} \#T  + \# S_\mathbb Q + \# \{p \in S_\mathbb Q \mid p \textrm{ split in }Q \} +1 \leq \frac{ (\#T -1)^2}{4} \end{equation} then $G$ is infinite.
\end{enumerate}
\end{lemma}

\begin{proof} For part (1), we consider the extension $\mathbb Q( \{\sqrt{q} \mid q \in T\})$ of $\mathbb Q$, which certainly contains $Q$. The extension $\mathbb Q( \{\sqrt{q} \mid q \in T\})/\mathbb Q$ is Galois with Galois group $(\mathbb Z/2\mathbb Z)^{ \#T}$, so the extension  $\mathbb Q( \{\sqrt{q} \mid q \in T\})/Q$ is Galois with Galois group the unique, up to isomorphism, index $2$ subgroup $(\mathbb Z/2\mathbb Z)^{ \#T-1}$ of $(\mathbb Z/2\mathbb Z)^{ \#T}$.

To check that $\mathbb Q( \{\sqrt{q} \mid q \in T\})/Q$ is everywhere unramified, it suffices to check that $Q(\sqrt{q})/Q$ is unramified for each $q$ as the composition of everywhere unramified extensions is everywhere unramified. Every ramified place of $Q(\sqrt{q})/Q$  must lie over a ramified place of $\mathbb Q(\sqrt{q})/\mathbb Q$. On the other hand, $Q(\sqrt{q})/Q=Q(\sqrt{ \prod_{q'\in T\setminus\{q\}} q'})/Q$ and hence every ramified place of $Q(\sqrt{q})/Q$  must lie over a ramified place of $\mathbb Q(\sqrt{ \prod_{q'\in T\setminus\{q\}} q'})/\mathbb Q$. No odd place is ramified in both since no prime divides both $q$ and $\prod_{q'\in T\setminus\{q\}} q'$ and $2$ is not ramified in both since only one of $q$ and $\prod_{q'\in T\setminus\{q\}} q'$ is not congruent to $1$ modulo $4$ since the number of elements of $T$ congruent to $3$ mod $4$ is odd. 

$\mathbb Q( \{\sqrt{q} \mid q \in T\})$ is certainly totally real. As a composition of quadratic extensions of $\mathbb Q$, the inertia degree of each prime of $\mathbb Q$ in $\mathbb Q( \{\sqrt{q} \mid q \in T\})$ is at most $2$. Since inertia degrees are multiplicative in field extensions, the inertia degree of each prime of $Q$ in $\mathbb Q( \{\sqrt{q} \mid q \in T\})$ is at most $2$, and is at most $1$ if the prime is inert over $\mathbb Q$.

For part (2), by part (1), the Galois group $(\mathbb Z/2\mathbb Z)^{ \#T-1}$ of  $\mathbb Q( \{\sqrt{q} \mid q \in T\})/Q$  is a quotient of $G$. It follows that $d(G) \geq \#T -1$.

For part (3), we first observe that $ \# S_\mathbb Q + \# \{p \in S_\mathbb Q \mid p \textrm{ split in }Q \} $ is the number of primes in $Q$ that lie over primes in $S_{\mathbb Q}$, as there are two primes lying over each split prime and one lying over each other prime. We next observe that $G$ is the quotient of the Galois group of the composition of all Galois extensions of $Q$, with degree a power of $2$, that are everywhere unramified and totally real, by the normal closure of the subgroup generate $ \# S_\mathbb Q  + \# \{p \in S_\mathbb Q \mid p \textrm{ split in }Q \} $ elements, those being the Frobenius elements of each prime lying over an inert prime of $S_{\mathbb Q}$ and the square of the Frobenius of each prime lying over another prime of $S_{\mathbb Q}$. Each time we quotient by the normal subgroup generated by an element by increases the minimal number of relations by at most $1$,  unless it is a generator, in which case it reduces the minimal number of generators by $1$ and does not increase the minimal number of relations. Thus it suffices to show that, for the Galois group $G'$ of the composition of all Galois extensions of $Q$, with degree a power of $2$, that are everywhere unramified and totally real, we have $r(G')\leq d(G')+2$. This follows from \cite[Theorem 10.7.12]{Neukirch2008} since in that theorem we may take $S =T=\emptyset$ since by convention $\mathbb C$ is a ramified extension of $\mathbb R$ and then $r=2$ and $\theta=1$, and the other terms vanish, so the Euler characteristic is at most $3$, and the Euler characteristic is $1+r(G')-d(G')$ so $r(G')-d(G')\leq 3-1=2$. 

For part (4), the Golod-Shafarevich theorem~\cite{GolodShafarevich} in its refined form due to Gasch\"utz and Vinberg~\cite{Vinberg1965,Koch1969} states that $G$ is infinite if $r(G) \leq \frac{ d(G)^2}{4}$ and thus if $r(G) -d(G) \leq \frac{d(G)^2}{4} - d(G)$, so by part (3), $G$ is infinite as long as 
\begin{equation}\label{sheared-gs} \# S_\mathbb Q + \# \{p \in S_\mathbb Q \mid p \textrm{ split in }Q \} +2  \leq  \frac{d(G)^2}{4} - d(G). \end{equation} 
\eqref{sheared-gs} is never satisfied if $d(G) <4$ as the left hand side is positive and the right hand side is not, and the right hand side of \eqref{sheared-gs} is increasing for $d(G) \geq 4$, so we may replace $d(G)$ on the right hand side of \eqref{sheared-gs} by the lower bound $\#T-1$ for $d(G)$ coming from part (2), obtaining that $G$ is infinite as long as
\[  \# S_\mathbb Q + \# \{p \in S_\mathbb Q \mid p \textrm{ split in }Q \} +2 \leq  \frac{ (\#T -1)^2}{4}  - (\#T -1),\] as desired.
\end{proof}

\begin{lemma}\label{field-existence} Let $T$ be a finite set of odd primes. Assume that the number of elements of $T$ which are congruent to $3$ modulo $4$ is odd. Let $S_{\mathbb Q}$ be a finite set of primes of $\mathbb Q$, and assume that each prime of $S_{\mathbb Q}$ is either inert in $\mathbb Q(\sqrt{q})$ for some $q\in T$ or congruent to $1$ mod $4$. Let $Q=\mathbb Q(\sqrt{\prod_{q\in T} q})$.

If \eqref{infinitude-criterion} is satisfied then there exist Galois extensions $F$ of $\mathbb Q$ of arbitrarily large degree that are totally real, such that for $K = F(\sqrt{-1})$ we have $\operatorname{rd}_{K/F} = \sqrt{ 4 \prod_{q\in T} q}$, and for each prime $p$ in $S_{\mathbb Q}$, every prime of $F$ lying over $p$ splits in the extension from $F$ to $K$, the ramification index of $p$ in $F$ is $2$ if $p\in T$ or $p=2$ and $1$ otherwise, and the inertia degree of $p$ in $F$ is at most $2$. \end{lemma}
 
 \begin{proof} Since $\prod_{q\in T} q$ is squarefree and congruent to $3$ mod $4$, we have $\Delta_Q = 4 \prod_{q\in T} q$.
 
 The group $G$ defined in Lemma \ref{group-infinite} is infinite since we have assumed that $S_\mathbb Q$ and $T$ satisfy the condition \eqref{infinitude-criterion} of Lemma \ref{group-infinite}(4). Since $G$ is a profinite group, it has arbitrarily large finite quotients that surject onto the quotient $\Gal(\mathbb Q( \{\sqrt{q} \mid q \in T\})/Q)$ of $G$ obtained from Lemma \ref{group-infinite}(1). Each is associated to a Galois extension $F'$ of $Q$ that is everywhere unramified and totally real, and such that for all primes $\mathfrak p$ of $Q$ lying over $p \in S_{\mathbb Q}$, the inertia degree of $\mathfrak p$ in the extension is at most $2$, and is at most $1$ if $p$ is inert in $Q$. The extension $F'$ need not be Galois over $\mathbb Q$, but if we let $F^*$ be the pullback of $F'$ along the unique nontrivial automorphism of $Q$ then $F^*$ has all the same properties as $F'$ and, since these properties are stable under composition, $F' F^*$ does as well, and $F' F^*$ is Galois. Thus, let $F = F' F^*$. The degree of $F$ is at least the degree of $F'$ and hence may be arbitrarily large.
 
To check $K/F$ is unramified at all finite places, we observe that $K = F(\sqrt{-1})$ is unramified over $F$ away from places above $2$, but also $K=F(\sqrt{ -\prod_{q\in T}q})$ and $-\prod_{q\in T}q$ is congruent to $1$ mod $4$ so $ K/F$ is unramified over $F$ at $2$. Since $F$ is unramified over $Q$, we have
\[ \frac{\Delta_K}{\Delta_F}  = \frac{\Delta_F^2}{\Delta_F}=\Delta_F = \Delta_Q^{[F:Q]}= \Delta_Q^{\frac{[F:\mathbb Q]}{2}}\]
so $\operatorname{rd}_{K/F} = \sqrt{ 4 \prod_{q\in T} q}$.

For each prime $p$ in $S_{\mathbb Q}$, if $p$ is inert in $\mathbb Q(\sqrt{q})$ for some $q\in T$ then $p$ has inertia degree divisible by $2$ in the extension $\mathbb Q( \{\sqrt{q} \mid q \in T\})/\mathbb Q$, as this contains $Q(\sqrt{q})$, and thus has inertia degree divisible by $2$ in $F/\mathbb Q$. Since the inertia degree of a composition of two extensions is the least common multiple of the inertia degrees, and $p$ has inertia degree at most $2$ in $\mathbb Q(\sqrt{-1})/\mathbb Q$, the inertia degree of $p$ in $K/\mathbb Q$ is equal to the inertia degree of $p$ in $F/\mathbb Q$. Combined with the fact that $K$ is unramified over $F$, we conclude that all primes lying over $p$ in $F$ are split in $K/F$.  If $p$ is congruent to $1$ mod $4$, then all primes lying over $p$ are split in $K/F$ since $p$ is already split in $\mathbb Q(i)$. 

Since $F/Q$ is unramified everywhere, the ramification index of $p$ in $F$ equals the ramification index of $p$ in $Q$, which is $2$ if $p\in T$ or $p=2$ and $1$ otherwise.

The inertia degree of $p$ in $F/\mathbb Q$ equals the inertia degree in $Q/\mathbb Q$ of $p$ times the inertia degree in $F/Q$ of a prime lying over $p$ in $Q$. The inertia degree in $Q/\mathbb Q$ of $p$ is $2$ if $p$ is inert in $Q$ and $1$ otherwise, and we have constructed $F$ so that the inertia degree in $F/Q$ of $p$ is $1$ if $p$ is inert in $Q$ and at most $2$ otherwise, so the inertia degree of $p$ in $F/\mathbb Q$ is at most $2$, as desired. \end{proof}

\begin{proof}[Proof of Theorem \ref{main}]

By Lemma \ref{field-existence}, we see that if $T$ is a finite set of odd  primes such that the number of elements of $T$ which are congruent to $3$ modulo $4$ is odd and $S_{\mathbb Q}$ is a finite set of primes of $\mathbb Q$, such that each prime of $S_{\mathbb Q}$ is either inert in $\mathbb Q(\sqrt{q})$ for some $q\in T$ or congruent to $1$ mod $4$, and satisfying \eqref{infinitude-criterion}, we can take
\begin{equation}\label{delta-specialized} \delta=  \frac{  \log \left(1-\frac{1}{R}\right) + \frac{\log \left(\frac{2\pi}{e} \right)}{2} +\sum_{p\in S_{\mathbb Q}} \frac{1}{4e(p) } \log (k(p)+1)  - \frac{\log (4 \prod_{q\in T}q )}{8} -\frac{\log( \log \sqrt{4 \prod_{q\in T}q})}{2} }{ \log \left( 2R \prod_{p \in S_\mathbb Q} p^{\frac{k(p) }{2e(p)}} +1 \right)}\end{equation}
for any $R>1$ in Proposition \ref{main-sequence-criterion}. 

If we take \[T = \{ 3,5, 7, 11, 13, 17,19, 23, 29, 31,37,41, 43\}\] and \[S_\mathbb Q =\{2,3,5,7,11,13, 17,19,23,29,  47, 71, 79, 97, 101,107, 109,139,151,163, 167, 179 \}\] then we have $\#T = 13$ and $\#S_\mathbb Q =22$. The number of elements of $T$ congruent to $3$ mod $4$ is $7$, which is odd. The $10$ smallest primes in $S_\mathbb Q$ are ramified in $Q$ and the $12$ largest primes in $S_\mathbb Q $ are inert in $Q$, as can be computed in Sage\cprotect\footnote{by running \verb|[kronecker(3*5*7*11*13*17*19*23*29*31*37*41*43,x) for x in|\\ \verb|[47,71,79,97,101,107,109,139,151,163,167,179]]|}. Thus no primes in $S_{\mathbb Q}$ are split in $Q$ and we have $\#T + \#S_{\mathbb Q} +1=36 = \frac{12^2}{4} = \frac{ (\#T-1)^2}{4}$ so \eqref{infinitude-criterion} is satisfied. The primes inert in $Q$ are automatically inert in $\mathbb Q(\sqrt{q})$ for some $q\in T$, and we can check that all the ramified primes are either congruent to $1$ mod $4$, or inert in $\mathbb Q(\sqrt{q})$ for $q=3,5,$ or $7$. Taking $R= 72$ and $k(2)=50, k(3)=31, k(5)=21, k(7)=17, k(11)=14, k(13)=13, k(17)=12, k(19)=11, k(23)=10, k(29)=10, k(47)=8,  k(71)=k(79)=k(97)=k(101)=k(107)=k(109)=7, k(139)=k(151)=k(163)=k(167)=k(179)=6$, the numerator in \eqref{delta-specialized} is $3.8822\dots$ and the denominator is $275.055\dots$ so we can take $\delta=.014114\dots$ in Proposition \ref{main-sequence-criterion}, giving the statement.\end{proof}

The values $k(p)$ were obtained by the formula $k(p)= \lfloor \frac{1}{p^{ t^{-1}}-1} \rfloor$ and $R$ by the formula $R=2t+1$, for $t=35.5$. These are the optimal values if the quantity to be optimized were $2c$ times the numerator of \eqref{delta-specialized} minus the denominator of \eqref{delta-specialized}, which if $c$ is chosen well should closely approximate the optimizers for \eqref{delta-specialized}, and $t$ was chosen by trying a few values and taking the best one. $T$ consists of the first few primes while $S_\mathbb Q$ consists of the first few ramified primes in $Q$ and the first few inert primes. None of these were optimized very carefully.

\section{Discussion}

We first explain in more detail how our arguments differ from those~\cite{machinewriteup} obtained by Lijie Chen using an internal OpenAI model, with a simplification suggested by Victor Wang:

In Lemma \ref{set-from-lattice}, we apply the probabilistic method to bound the ratio $\# \{(v_1,v_2)\in U^2 \mid \abs{v_1-v_2}=1 \}/\#U$ instead of the value $\# \{(v_1,v_2)\in U^2 \mid \abs{v_1-v_2}=1 \}$.  Note that the improvement here is proportional to the ratio between the average value of $\#U$ and the upper bound on $\#U$ obtained in the argument, and the upper bound on $\#U$ obtained in either argument is much larger than the average value.

In Lemma \ref{lattice-from-ideal}, we work with an ideal instead of simply the ring of integers.

In Lemma \ref{ideal-from-primes}, the pigeonhole argument, we make several improvements. First, we do not require the prime ideals $\mathfrak p$ of $F$ to be split over $\mathbb Q$. This assumption was made in the OpenAI team's proof  and likely led to significant loss as it means one can only use primes split in the field extension $\mathbb Q(\{ \sqrt{q}\mid q \in T\})$, and the smallest primes split in this extension are typically exponentially large in $T$. Second, we introduce the parameter $k(\mathfrak p)$ that allows taking powers of the prime ideal, allowing one to get more leverage out of fewer primes. (This idea was also used to simplify the argument in the human-created writeup~\cite{humanwriteup}.) Third, the fact that we work with a fractional ideal means we don't have to embed the optimal fractional ideal into a principal ideal, incurring further loss. Fourth, once we obtain a large number of elements whose norms to $F$ generate the same ideal, we do not divide them by their complex conjugates to obtain elements with the same norm, but instead use the unit group to control them, in the form of Lemma \ref{modified-class-group}. Fifth, we use the relative class group instead of the class group for pigeonholing.

In Lemma \ref{class-group-bound}, we use \cite{Louboutin2000} instead of the Minkowski bound to get an improved bound on the (relative) class number. 

In Lemma \ref{group-infinite}, we work with the maximal unramified extension of $Q$ instead of the maximal extension of $\mathbb Q$ ramified only at the primes in $T$, which has substantially smaller root discriminant. We also avoid the assumption that $p\in S_\mathbb Q$ are congruent to $1$ mod $4$, which was used only to ensure that $\mathbb Q(\sqrt{p})$ is unramified at $2$, so $Q(\sqrt{p})/Q$ is unramified, by instead ensuring that $Q$ is ramified at $2$.

The flexibility to use primes that are not split is used heavily in the proof of Theorem \ref{main}. In fact, every prime in $S_{\mathbb Q}$ is either inert or ramified in $Q$.

\begin{remark} We explain how it is possible to prove a version of Theorem \ref{main} which shows that a set $U$ exists with $\#U=n$ and $\# \{ (v_1,v_2) \in U \mid \abs{v_1-v_2}=1\} \geq n^{1+\delta}/O(1)$ for many $n\leq N$. To do this, observe that the construction of Proposition \ref{main-sequence-criterion} produces, for some $X > 1$, for $d$ the degree of a suitable field $F$, a set with \[\#U \leq X^d \textrm{ and } \frac{ \# \{ (v_1,v_2) \in U \mid \abs{v_1-v_2}=1\} }{\#U }\geq X^{\delta d }/O(1).\] Next observe that we can adapt Lemma \ref{field-existence} to produce fields whose degree is any fixed sufficiently large power of $2$, since $\Gal(F/\mathbb Q)$ is a $2$-group, $2$-groups always contain central elements of order two in any nontrivial subgroup, and taking the fixed field of a central element of order $2$ in  $\Gal(F/\mathbb Q)$ that lies in the kernel of the natural map to $\Gal(\mathbb Q( \{\sqrt{q} \mid q \in T\})/\mathbb Q)$ produces a field of degree smaller by a factor of $2$ that has all the same properties. Thus for $N$ sufficiently large we can take 
\[d= 2^{\bigl\lfloor \frac{ \log \left( \frac{\log N}{\log X}\right)}{\log 2} \bigr \rfloor}\] so that $\#U \leq X^d \leq N$ and then we have 
\[ \frac{\# \{ (v_1,v_2) \in U \mid \abs{v_1-v_2}=1\} }{\#U} \geq \frac{X^{\delta  d}}{O(1)} \geq  \frac{X^{\delta \cdot  2^{\bigl\lfloor \frac{ \log \left( \frac{\log N}{\log X}\right)}{\log 2} \bigr\rfloor }} }{O(1)}\geq \frac{X^{ \delta \cdot 2^{ \frac{ \log \left( \frac{\log N}{\log X}\right)}{\log 2}  -1} }}{O(1)}= \frac{N^{ \frac{\delta}{2}}}{O(1)}\]
so $\#U \geq \frac{N^{\frac{3\delta}{2}}}{O(1)}$ by \cite{Spencer1983}. So we can always take some $U$ of size $n$ for some $n \in [N^{\frac{3\delta}{2}}/O(1),N]$. Furthermore, a random subset of $U$ consisting of at least half the elements will again have many pairs of points at unit distance, so we show such sets exist for at least $N^{\frac{3\delta}{2}}/O(1)$ values of $n\leq N$.\end{remark}

\begin{remark} We discuss a final question: What are the limits for the exponents that can be obtained by this argument? To make this into a precise question, set aside the question of improvements to Lemmas \ref{set-from-ideal} and \ref{ideal-from-primes}. Combining Lemmas \ref{set-from-ideal} and \ref{ideal-from-primes}, one obtains an exponent of
\begin{equation}\label{bound-with-class-number} 1+ \frac{ \log \left(1-\frac{1}{R}\right) +  \frac{1}{2d} \log \left( \frac{ \prod_{\mathfrak p \in S_F} (k(\mathfrak p)+1  )}{ 2^d h^{-}(K)} \right) }{\log \left( 2R\prod_{\mathfrak p\in S_F} \# (\mathcal O_F/\mathfrak p)^{\frac{k(\mathfrak p)}{2d}} +1 \right)}   \end{equation}
given a totally real field $F$ of degree $d$, CM extension $K$ of $F$, set $S_F$ of primes of $F$ consisting only of primes split in $K$, function $k$ from $S_F$ to positive integers, and real $R>1$. One precise formulation of the question is to ask for the best possible upper bound on \eqref{bound-with-class-number}, for arbitrary $F,K, S_F, k, R$ satisfying these conditions, for large $d$. While our argument makes use of upper bounds for the relative class number, solving this problem requires lower bounds for the relative class number, which are more subtle due to the potential for Siegel zeros. (Of course, it is also interesting to obtain bounds for \eqref{bound-with-class-number} conditional on the generalized Riemann hypothesis, as they will still be upper bounds for what can be obtained from the above argument unless the generalized Riemann hypothesis is disproved.)

Even using the trivial lower bound $1$ for the relative class number, one show that \eqref{bound-with-class-number} is at most $1+ \frac{1}{4.116}<\frac{4}{3}$ as in Proposition \ref{upper-bound} below. The goal would be to improve this bound. Since \eqref{bound-with-class-number} is large only when $F$ has many primes of small norm that split in $K$, it suffices to give lower bounds for the relative class number in this setting, which may be easier, either using $L$-function methods or the methods of Ellenberg-Venkatesh~\cite{Ellenberg2007}. \end{remark}
 
 \begin{prop}\label{upper-bound} For any positive integer $d$, real $R>1$, totally real field $F$ of degree $d$, CM field $K/F$, set of primes ideals $S_F$ of $\mathcal O_F$, and function $k$ from $S_F$ to positive integers, the quantity \eqref{bound-with-class-number} is at most $1+\frac{1}{4.116} = 1.24295 \dots \dots$.\end{prop} 
 
\begin{proof}  For any $c>0$ we have 
\[c \left(\log \left(1-\frac{1}{R}\right) +  \frac{1}{2d} \log \left( \frac{ \prod_{\mathfrak p \in S_F} (k(\mathfrak p)+1  )}{ 2^d h^{-}(K)} \right) \right) -   \log \left( 2R\prod_{\mathfrak p\in S_F} \# (\mathcal O_F/\mathfrak p)^{\frac{k(\mathfrak p)}{2d}} +1 \right) \]
\[ \leq c \left(\log \left(1-\frac{1}{R}\right) +  \frac{1}{2d} \log \left( \frac{ \prod_{\mathfrak p \in S_F} (k(\mathfrak p)+1  )}{ 2^d h^{-}(K)} \right) \right) - \log \left( 2R\prod_{\mathfrak p\in S_F} \# (\mathcal O_F/\mathfrak p)^{\frac{k(\mathfrak p)}{2d}} \right) \]
\[ =c \frac{\log(2)}{2} - \log(2) + c\log \left(1-\frac{1}{R}\right) - \log(R)   - c \frac{\log h^{-}(K)}{2d} +\frac{1}{2d}  \sum_{\mathfrak p\in S_F} (c \log  (k(\mathfrak p)+1  )- k(\mathfrak p) \log  \# (\mathcal O_F/\mathfrak p)).\]

We bound the terms individually. We have \[ c\log \left(1-\frac{1}{R}\right) - \log(R) \leq c \log \left( 1- \frac{1}{c+1}\right)- \log(c+1)\] since the second derivative of the left hand side is maximized so the unique maximizer is obtained at point $R=c+1$ where the derivative vanishes.  We use the trivial bound $c \frac{\log h^{-}(K)}{2d}\geq 0$. Finally we have
\[ \sum_{\mathfrak p\in S_F} (c \log  (k(\mathfrak p)+1  )- k(\mathfrak p) \log  \# (\mathcal O_F/\mathfrak p))\leq  \sum_{\mathfrak p \textrm{ prime of }\mathcal O_F } \max_{k \geq 0}  (c \log  (k+1  )- k  \log  \# (\mathcal O_F/\mathfrak p)) \]
\[ = \sum_{p \textrm{ prime}} \sum_{\substack{\mathfrak p \textrm{ prime of }\mathcal O_F \\ \mathfrak p\mid p}}  \max_{k \geq 0}  (c \log  (k+1  )- k  \log  \# (\mathcal O_F/\mathfrak p)) \leq  \sum_{p \textrm{ prime}} \sum_{\substack{\mathfrak p \textrm{ prime of }\mathcal O_F \\ \mathfrak p\mid p}}    \max_{k \geq 0}  (c \log  (k+1  )- k  \log   p)  \]\[ \leq \sum_{p \textrm{ prime}} d  \max_{k \geq 0}  (c \log  (k+1  )- k  \log   p)  \]
since there are as most $d$ primes of $F$ lying over each prime $p$ and each has a residue field of size at least $p$. Combining these observations, we obtain
\[c \left(\log \left(1-\frac{1}{R}\right) +  \frac{1}{2d} \log \left( \frac{ \prod_{\mathfrak p \in S_F} (k(\mathfrak p)+1  )}{ 2^d h^{-}(K)} \right) \right) -   \log \left( 2R\prod_{\mathfrak p\in S_F} \# (\mathcal O_F/\mathfrak p)^{\frac{k(\mathfrak p)}{2d}} +1 \right) \]
\[\leq  c \frac{\log(2)}{2} - \log(2) + c \log \left( 1- \frac{1}{c+1}\right)- \log(c+1) + \frac{1}{2} \sum_{p \textrm{ prime} } \max_{k \geq 0}  (c \log  (k+1  )- k  \log   p) .\]
Setting $c =4.116$, this bound evaluates to $-.00105 \dots <0$, with the maximum value of $k$ equal to $5$ for $p=2$, $3$ for $p=3$, $2$ for $p=5$, $1$ for $p=7,11, 13,17$, and $0$ for all other values of $p$. Hence the numerator in \eqref{bound-with-class-number} is at most $4.1$ times the denominator and thus the fraction is at most $\frac{1}{4.1}$, as desired.\end{proof}

\bibliographystyle{plain}

\bibliography{references}

\end{document}